\documentclass[11pt]{article}
\def\RR{\mathbb{R}}

\def\uu{\text{\bf u}}
\def\vv{\text{\bf v}}

\def\WW{\text{\bf W}}
\def\ff{\text{\bf f}}

\usepackage[latin1]{inputenc}
\usepackage[USenglish,english,american]{babel}
\usepackage{amsfonts,setspace}
\usepackage{framed}

\usepackage{cancel}
\usepackage{amssymb}
\usepackage{anysize} 
\usepackage{multicol}
\usepackage[dvipsnames]{xcolor}
\usepackage{graphicx}
\usepackage{float}
\usepackage{bbm}
\usepackage{amsmath}
\usepackage{mathrsfs}
\usepackage{mathtools}
\usepackage{amsthm}
\usepackage{hyperref}
\usepackage{geometry}
\hypersetup{
    colorlinks,
    citecolor=blue,
    filecolor=magenta,
    linkcolor=red,
    urlcolor=black
}

\theoremstyle{theorem}
\newtheorem{theorem}{Theorem}
\theoremstyle{corollary}
\theoremstyle{proposition}
\newtheorem{proposition}{Proposition}
\theoremstyle{lemma}
\newtheorem{lemma}{Lemma}
\theoremstyle{definition}
\theoremstyle{condition}
\newtheorem{condition}{Condition}
\theoremstyle{Assumption}
\theoremstyle{remark}
\newtheorem{remark}{Remark}
\theoremstyle{claim}

\theoremstyle{example}

\theoremstyle{obs}

\usepackage{color}

\usepackage{color}

\usepackage[normalem]{ulem}

\begin{document}


\begin{center}
\begin{Large}{\bf Thermoacoustic tomography for an integro-differential wave equation modeling attenuation}\end{Large}\\
\vspace{1em}

{\bf Benjam\'in Palacios$^1$ \& Sebasti\'an Acosta$^2$}
\end{center}
{\small 
\begin{tabular}{rl}
$^1$ &\hspace{-3mm}Department of Mathematics, University of Washington, Seattle, WA, USA\\ 
&\hspace{-3mm}{\it Email address:} \texttt{bpalacio@uw.edu}\\
$^2$ &\hspace{-3mm}Department of Pediatrics - Cardiology, Baylor College of Medicine, TX, USA\\
&\hspace{-3mm}{\it Email address:} \texttt{sacosta@bcm.edu}
\end{tabular}
}

\renewcommand{\abstractname}{} 
\begin{abstract} 
\noindent{\sc Abstract.} In this article we study the inverse problem of thermoacoustic tomography (TAT) on a medium with attenuation represented by a time-convolution (or memory) term, and whose consideration is motivated by the modeling of ultrasound waves in heterogeneous tissue via fractional derivatives with spatially dependent parameters. Under the assumption of being able to measure data on the whole boundary, we prove uniqueness and stability, and propose a convergent reconstruction method for a class of smooth variable sound speeds. By a suitable modification of the time reversal technique, we obtain a Neumann series reconstruction formula.

\end{abstract}

\noindent{\it Key words:} Multiwave imaging; wave equation; integro-differential equations; attenuation; memory\\


\section{Introduction}

It is well known that for biological tissues the attenuation of acoustic waves is frequency-dependent.  One way to model this attenuation is to use fractional time derivatives and consequently the representation of the propagation of ultrasound waves by integro-differential equations. Examples of this modeling are frequency power-law attenuation or fractional Szabo models (see for instance \cite{1994Szabo,chen2003modified,2006Patch,2010TreebyCox,2011Roitner,2012PATatt,2012PATatt,2011Ammari}) where the traveling wave may be assumed to satisfy an equation of the form 
\begin{equation*}\label{TATfracderiv}
\gamma^{-2} \partial_t^2u - \Delta u + \beta \partial^{k+\alpha}_{t}u = F(t,x),\quad \text{for some}\quad\alpha\in(0,1),\; k=1,2,
\end{equation*}
and where the fractional derivative term can be written as a convolution in time $\beta(x)\partial^{k+\alpha}_{t}u = \int^t_{-\infty}\Psi_\alpha(t-s,x)\partial^{k+1}_su(s,x)ds$. Assuming, as in thermoacoustics, that the wave field vanishes for negative times, and provided that the kernel is bounded and regular enough, we can perform integration by parts and write the previous integral as a convolution of $u$ with a different kernel, plus time-derivatives of $u$ up to order two. In the case $k=2$, the sound speed is perturbed resulting in a different speed $c^{-2} = \gamma^{-2} + \beta\Psi_\alpha(0)$, which requires conditions on $\beta$ and $\Psi_\alpha$ in order to get an effective wave speed $c>0$. We point out there is a recent definition for derivatives of fractional order which employs such continuous and bounded kernels \cite{2015Caputo}.

In the present article, we study the inverse problem of finding the initial source $f$ in an attenuating medium, provided boundary data $u|_{[0,T]\times\partial\Omega}$ and where the acoustic wave $u$ is assumed to satisfies the system
\begin{equation}\label{TAT0}
\left\{\begin{array}{ll} \partial_t^2u - c^2\Delta u + a \partial_{t}  u + bu + \int^t_{-\infty}\Phi(t-s,x)u(s,x)ds = \delta'(t)f(x),&\in\RR\times\RR^{n}\\ u(t,x)=0,&t<0.\end{array}\right.
\end{equation}
We suppose $a,b,c\in C^\infty(\RR^n)$, $\Phi\in C^2(\RR^{n+1})$, $a,b\geq 0$, $c_{0}^{-1} \geq c \geq c_{0} > 0$, and for a fixed open bounded set $\Omega\subset\RR^n$ with smooth boundary, we suppose $a= b=c-1=\Phi = 0$ in $\RR^n\backslash\bar{\Omega}$. 
We shall use the following notation throughout the paper: 
$$P_\Phi := \partial_t^2-c^2\Delta + a\partial_t + b +\Phi*\cdot,\quad \Phi*u = \int^t_0\Phi(t-s,x)u(s,x)ds.$$
Then $P_\Phi = \partial_t^2-c^2\Delta$ outside the domain of interest $\Omega$. 
The Cauchy problem associated with \eqref{TAT0} is
\begin{equation}\label{TAT1}
\left\{\begin{array}{ll} P_\Phi u = 0,&(t,x)\in(0,\infty)\times\RR^{n},\\ u|_{t=0}=f,\\ \partial_tu|_{t=0}=-af,\end{array}\right.
\end{equation}
since any solution of \eqref{TAT1} extended by zero to $(-\infty,0)\times\RR^n$ is a solution of \eqref{TAT0}. Indeed, given a smooth solution $u$ of \eqref{TAT1} we consider $H(t)u(x,t)$ where $H(t)$ is the Heaviside function. Then, we can pull out the Heaviside function from the convolution since it integrates on the interval $(0,t)$, thus we get
\begin{align*}
P_\Phi(Hu) &= u\delta' + 2(\partial_tu)\delta  + au\delta + (P_\Phi u)H 
\end{align*}
with the last term 
vanishing since $P_\Phi u = 0$. 
For an arbitrary test function $\phi\in C^\infty_c(\RR^{n+1})$ we have the following,
\begin{align*}
\langle P_\Phi(Hu) , \phi\rangle &= \int_{\RR^{n}}\big[-(\partial_t u)\phi - u(\partial_t\phi) + 2(\partial_tu)\phi + au\phi\big]\big|_{t=0}dx\\
&= -\int_{\RR^n}u\partial_t\phi|_{t=0}dx\\
&=\langle f\delta',\phi\rangle, 
\end{align*}
which is the same as problem \eqref{TAT0}.




The thermoacoustic tomography problem in a medium with convolution-type attenuation can be modeled by the following initial value problem (IVP):
\begin{equation}\label{TAT2}
\left\{\begin{array}{ll} P_\Phi u(t,x) = 0,&(t,x)\in(0,T)\times\RR^{n}\\ u|_{t=0}=f,\\ \partial_tu|_{t=0}=-af,\end{array}\right.
\end{equation}
where we aim to recover the initial source $f$ from boundary measurements $u|_{(0,T)\times\partial\Omega}$, assuming the waves propagate freely in the space, that is, we suppose the boundary of $\Omega$ does not interact with the outgoing waves. This last assumption has been considered for instance in \cite{TATbrain,Homan,2016Palacios}.

The problem of thermoacoustic tomography has been broadly studied by many authors. Several reconstruction methods have been proposed for homogeneous media \cite{Finch_Rakesh_2004,2009Kunyansky,2007Finch,Kunyansky2011,WangAnastasio2012,Natterer2012,
Palamodov2012,Haltmeier2014}, and also for heterogeneous media \cite{Anastasio-2007, hristova2008reconstruction, 2009Hristova, TAT, TATbrain, Tittelfitz-2012, HuangWaNie2013, 2012PATatt, 2015AcostaMontalto, 2015StefanovYang,2015Nguyen,2016Oksanen,2015StefanovYang}. See also the reviews \cite{Agranovsy-2007,Kuchment2011,Bal-2011} for additional references. The theoretical analysis of the so-called {\it time reversal} method has gained considerable attention in the past few years, mainly due to the work of Stefanov and Uhlmann in \cite{TAT,TATbrain}. In its initial formulation, the time reversal technique gives an approximate solution that converges to the exact one as the observation time increases. The problem of recovering the initial source for optimally short measurement time was solved in \cite{TAT} for variable sound speed employing techniques from microlocal analysis. 

Recently, the focus of the mathematical analysis has been placed on extensions in the following two areas. First, there is the problem of accounting for attenuating media. Homan in \cite{Homan} gave a first extension of Stefanov and Uhlmann's work in this direction by considering the damped wave equation with sufficiently small damping coefficients for the time reversal method to work. In the complete data case, those results were extended to more general damping coefficients in \cite{2016Palacios}. In \cite{2016AcostaMontalto} the authors addressed the TAT problem with thermoelastic attenuation. Second, recent publications have addressed the TAT problem in enclosed domains to model the interaction of acoustic waves with reflectors and sensors. The advantages of working with this setting is that it naturally allows to consider partial data and the inverse problem is closely related with boundary control theory. See for instance \cite{2015AcostaMontalto,2015StefanovYang,2015Nguyen,2016Oksanen}. 

This article falls in the first group. As far as the authors know, the TAT inverse problem with attenuation of integral type and variable sound speed has not been fully considered in the literature from an analytical point of view. From a heuristic point of view, some advances have been made. For the case of constant wave speed and constant coefficient of attenuation, Modgil et al. \cite{modgil2012} designed a method based on relating the unattenuated wave field to the attenuated wave field via an integral operator and its subsequent inversion using a singular value decomposition. Treeby et al. \cite{2010TreebyCox,treeby2010} proposed a reconstruction based on time reversal and the $k$-space computational method. Attenuation compensation was achieved by separating the absorbing and dispersion terms in the wave equation, and reversing the sign of the absorbing coefficient during the time reversal. This method was modified in \cite{2012PATatt} where the coefficient of attenuation was allowed to vary within the region of interest, but the exponent of the power-law attenuation was still assumed to be constant. However, in some practical settings such as in the presence of bone and soft--tissue, the domain exhibits regions of varying power-law exponents. An appropriate method needs to be devised to avoid blurring and distortions in the reconstruction. Our work is a step in that direction, where the coefficients $a,b,c$ and the kernel $\Phi$ in (\ref{TAT0}) are allowed to vary, which effectively accounts for power-law attenuation of spatially varying exponent.

Considering attenuation terms of integral type brings some difficulties to the analysis on the propagation of waves. In particular, the equation is no longer reversible and local in time and consequently it is not possible to use techniques such as Tataru's unique continuation to get uniqueness, at least not in a direct way. Moreover, the microlocal properties of this type of integro-differential operators are not well understood. Nevertheless, it is possible to exploit the fact that an integral term of the sort considered here only presents a compact perturbation of the differential operator. This article can be view as a first attempt to understand the TAT problem in media with memory/attenuation coefficients that vary in space. A subsequent step would be to tackle viscoelastic models, and singular kernels as in the standard definition of fractional derivatives. 

The paper is structured as follows. In the next section we set the framework under which our analysis is based, such as the well-posedness of the direct problem, the energy space of initial conditions and the hypothesis on the attenuation parameters, namely the damping coefficient and the attenuation kernel. In Section 3 we prove two uniqueness results. The first one is a sharp result on uniqueness for the thermoacoustic inverse problem assuming the distance function from the boundary allows us to foliate the interior of the domain by strictly convex surfaces. In particular we require $\partial\Omega$ to be strictly convex. The second main theorem of this section, which does not require convexity of the boundary, assumes that the sound speed satisfies a frequently used condition related with the convexity of the euclidean spheres in the metric induced by the sound speed. The stability question is addressed in Section 4 and we devote Section 5 to show the existence of a Neumann series reconstruction formula. 

\section{Preliminaries}

\subsection{Direct problem}

Let $U\subset\RR^n$ be an open bounded set with smooth boundary, $u_0\in H^1_0(U)$, $u_1\in L^2(U)$ and $F\in L^2([0,T]; L^2(U))$. We say $u$ is a generalized solution of 
\begin{equation}\label{wp_eq}
P_\phi u = F \text{ in }[0,T]\times U,\quad u|_{[0,T]\times\partial U}=0,\quad u(0) = u_0,\;u_t(0) = u_1,
\end{equation}
if $u\in L^2([0,T];H^1_0(U))$, $u_t\in L^2([0,T];L^2(U))$, $u_{tt}\in L^2([0,T];H^{-1}(U))$ and 
\begin{equation}\label{def:weak_sol}
\langle c^{-2}u_{tt},\varphi\rangle + B(u,\varphi) = (c^{-2}f,\varphi)\quad \forall \varphi\in C^\infty_0(U)\text{ and for a.e. }t\in[0,T]
\end{equation}
where $\langle\cdot,\cdot\rangle$ and $(\cdot,\cdot)$ stand for the duality product of $H^{-1}$ and $H^1_0$, and the $L^2$ inner product respectively, and $B(\cdot,\cdot)$ is the bilinear form given by
$$B(u,\varphi) = (\nabla u,\nabla\varphi) + (ac^{-2}u_t,\varphi)+ (bc^{-2}u,\varphi) + (c^{-2}\Phi*u,\varphi).$$

The well-posedness follows from Theorems 2.1 and 2.2 in \cite{1970Dafermos}.
We refer to the appendix for a complete proof. In our case, by finite speed of propagation we can take $U$ to be a large ball containing $\Omega$ to ensure we have null Dirichlet conditions. 

\subsection{Energy space and positive-definite kernels}

Given a domain $U\subseteq\RR^n$ and a scalar function $u(t,x)$, we define the local energy of $\uu = [u,u_t]$ at time $t$ as
$$E_U(\uu(t)) = \int_{U}(|\nabla_x u|^2 + b|u|^2 + c^{-2}|u_t|^2)dx.$$
In order to give problem \eqref{TAT2} a physical sense we need to assume some conditions on the attenuation terms since such system must satisfies that its energy decreases over time. The previous is achieved for instance if $a(x)\geq 0$ and the kernel $\Phi$ is positive-definite, this is $\int^T_0(\Phi*y)ydt\geq 0$ for all $y\in C([0,T])$. We then assume the following:
\begin{condition}\label{cond_att}
\begin{equation}\label{hyp_att}
a(x)\geq 0\quad\text{and}\quad (-1)^j\partial^j_t\Phi (t,x) \geq 0, \;\forall t\geq0,\;x\in\RR^n,\; j = 0,1,2.
\end{equation}
\end{condition}

The previous condition guarantees the positive-definiteness of the kernel as shown in \cite{1972MacCamy} and \cite{2016IPwave_memory}. Moreover, if we define 
\begin{equation}\label{def:kernel2}
\Psi(t,x) := -\int^\infty_t\Phi(s,x)ds,
\end{equation}
it turns out that $-\Psi$ is also a positive-definite kernel since it satisfies the same condition as $\Phi$.
\begin{remark}
An example of a kernel satisfying Condition \ref{cond_att} is $\Phi(t,x) = q(x)e^{-\alpha t}$, for some positive function $q \in C(\RR^n)$ and $\alpha > 0$. In the recent article \cite{2015Caputo}, the authors introduce a new definition for fractional derivatives whose kernel is of the form $e^{-\alpha t}$. As a consequence, the analysis carried out in this paper might be applied to fractional models of wave propagation following this new definition of fractional derivatives.
\end{remark}

Under Condition \ref{cond_att} we define the extended energy functional at time $\tau>0$, analogously as in \cite{2016Palacios}, to be
\begin{equation}\label{ext_eng}
\mathcal{E}_{U}(u,\tau) = E_{U}(\uu(\tau)) +2\int_{[0,\tau]\times U}\hspace{-.5em}ac^{-2}|u_t|^2dxdt+ 2\int_{[0,\tau]\times U}\hspace{-.5em}c^{-2}\left(\Phi*u_t\right)u_tdxdt,
\end{equation}
where the last two terms take into account the portion of the energy that is lost due to the attenuation process. If we set $U=\RR^n$, or by finite propagation speed we take $U$ equal to any sufficiently large ball, in the interval $[0,T]$ the former energy functional $E_U$ is non-increasing since we get
$$\frac{d}{dt}E_U(\uu(t)) = - 2\int_{[0,\tau]\times U}\hspace{-.5em}ac^{-2}|u_t|^2dxdt- 2\int_{[0,\tau]\times U}\hspace{-.5em}c^{-2}\left(\Phi*u_t\right)u_tdxdt\leq 0,$$
and integrating in time we deduce that the extended functional is conserved.

We adopt the same functional framework as in previous articles related to thermoacoustic tomography.
The energy space $\mathcal{H}(U)$ of initial conditions is defined to be the completion of $C^\infty_0(U)\times C^\infty_0(U)$ under the energy norm
$$\|\ff\|^2_{\mathcal{H}(U)} = \int_{U}(|\nabla_x f_1|^2 + c^{-2}|f_2|^2)dx.$$
with $\ff = [f_1,f_2]$. We also let $H_D(U)$ denote the completion of $C^\infty_0(U)$ under the norm
$$\|f\|^2_{H_{D}(U)} = \int_{U}|\nabla_x f|^2dx.$$
Notice that $\mathcal{H}(U) = H_D(U)\oplus L^2(U;c^{-2}dx)$ with the latter space denoting the $L^2$ functions under the weight $c^{-2}dx$.


Denoting by $\Omega$ the region of interest and $\Sigma = [0,T]\times\partial\Omega$, we introduce the measurement operator $\Lambda_\Phi:\mathcal{H}(\Omega)\ni \ff \mapsto u|_{\Sigma}\in H^1(\Sigma)$, where $u$ satisfies \eqref{TAT2}.

\section{Uniqueness}

The first main result of this section, Theorem \ref{uniqueness}, is a uniqueness theorem for the full data case under a particular foliation condition. We work in this part with the more general hyperbolic operator \eqref{dif_op} associated to a Riemannian metric $g$. We then, assuming $g=c^{-2}dx^2$, provide a condition for the sound speed that guaranteed the existence of a particular foliation suitable for uniqueness. This is our second main result, Theorem \ref{uniqueness2}. 

Foliation conditions go back to the work on seismology of Herglotz, Weichert and Zoeppritz at the beginning of the 20th century (see \cite[\S 6]{Uhlmann2016bdry_rig} and reference therein), and have been reintroduced in the literature very recently in \cite{2013StefanovUhlmann,Uhlmann2016bdry_rig}. This type of assumptions seem to be the natural conditions under one could expect to propagate information from the exterior toward inside the domain. In particular, it has been applied before in the thermoacoustic setting \cite{2013StefanovUhlmann} to prove uniqueness for the inverse problem of recovering the sound speed (assuming the initial condition is known).

Theorem \ref{uniqueness} is a direct consequence of \cite[Theorem 1]{2007Bukhgeim}, a unique continuation result for hyperbolic equations with a memory term. For the sake of simplicity, the authors proved such result for the wave equation in an Euclidean metric. Nevertheless, in our work we need the full strength of such unique continuation, thus we have included a brief proof in the general case of waves in a general Riemannian setting. We point out that a similar method to the one used in \cite{2007Bukhgeim} was also applied in \cite{2013StefanovUhlmann} to obtain uniqueness for very general foliations and partial data. It was of fundamental importance in such proof the possibility of using a partial boundary unique continuation result independent of the foliation (see \cite[Proposition 1]{2013StefanovUhlmann}). In contrast, in our context is precisely the unique continuation the result we need to prove, which in this case is deeply linked with the foliation. As a consequence we only treat the full data case. Modifying a bit the method mentioned above we are able to prove a second uniqueness result, Theorem \ref{uniqueness2}, in the case the sound speed satisfies an specific condition.


\begin{theorem}\label{uniqueness}
Let $\Omega$ be a bounded open subset of $\RR^n$ with $\partial\Omega$ smooth and strictly convex for a Riemannian metric $g$. Let $T>0$ be such that $x^n=\text{\rm dist}(x,\partial\Omega)$ is a smooth function in $\Omega$ with non-zero differential for $0\leq x^n\leq T$ and its level surfaces $\{x^n=s\}$, for $0\leq s\leq T$, are strictly convex for the metric $g$ as well. If $f\in H_{D}(\Omega)$ is such that $\Lambda_\Phi\ff=0$ with $\ff = (f,-af)$, then $f=0$ in $\{x\in\Omega: \text{\rm dist}(x,\Omega)<T\}$. In particular, if $T\geq T_1(\Omega):= \sup_{x\in\Omega}\text{\rm dist}(x,\partial\Omega)$, then $f\equiv 0$.
\end{theorem}

\begin{remark}
Under the above hypothesis, this result presents an improvement in the condition imposed on $T$ for uniqueness in the damped wave equation ($T>2T_1(\Omega)$ in \cite[Theorem 3.1]{Homan}). 
\end{remark}

Let $Q = (0,T)\times \Omega$ and $x^n = \text{dist}(x,\partial\Omega)$ the signed distance function defined in a neighborhood of the boundary and such that $\Omega$ and $\partial\Omega$ are characterized respectively by $x^n>0$ and $x^n=0$. We define the following weight function
\begin{equation}\label{varphi}
\varphi(x,t) = (R-x^n) - \alpha t^2 - r^2,
\end{equation}
which is invariantly defined for any local coordinates $(x^1,...,x^{n-1})$ in $\partial\Omega$.
Here $\alpha = \alpha(\Omega,g)>0$ is sufficiently small and $R,r>0$ will be chosen large and close to each other. For $\epsilon\geq 0$ we also consider the sets
\begin{eqnarray}\label{set1}
Q(\epsilon) &=& \{(t,x)\in Q: \varphi(x,t)>\epsilon\},\\
\Omega(\epsilon) &=& \{x\in\Omega:(R-x^n)^2>r^2+\epsilon\}. \label{set2}
\end{eqnarray}
By taking $r$ close to $R$, the set $Q(0)$ is a small neighborhood of $\{0\}\times\partial\Omega$ inside $Q$. 

We recall that in boundary normal coordinates, a Riemannian metric $g$ takes the form
\begin{equation}\label{metric_nbc}
\tilde{g}_{\alpha,\beta}(x',x^n)\text{d}x^\alpha \text{d}x^\beta + (\text{d}x^2)^2,
\end{equation}
for $\alpha,\beta\leq n-1$. We denote $\tilde{g} = (\tilde{g}_{\alpha\beta}(x))$. Moreover, the strictly convexity of the level surfaces $\{x^n=s\}$ translates into 
\begin{equation*}\label{convexity_vec}
\Pi(v,v) = \left(-\frac{1}{2}\frac{\partial \tilde{g}_{\alpha\beta}}{\partial x^n}\right)v^\alpha v^\beta \geq \kappa_s|v|^2_{\tilde{g}},\quad\forall v\in T\{x^n=s\},
\end{equation*}
with $\kappa_s>0$ the smallest eigenvalue (principal curvature) of the second fundamental form $\Pi$ in $\{x^n=s\}$, where $R_{s}=\kappa_s^{-1}$ can be think as the largest curvature radius of $\{x^n=s\}$.
The analogous condition for convectors follows from the natural isomorphism $\xi_i = g_{ij}(x)v^j$ and reads
\begin{equation}\label{convexity_covec}
\Pi(\xi,\xi) = \left(\frac{1}{2}\frac{\partial \tilde{g}^{\alpha\beta}}{\partial x^n}\right)\xi_\alpha \xi_\beta \geq \kappa_s|\xi|^2_{\tilde{g}},\quad\forall \xi\in T^*\{x^n=s\}.
\end{equation}

In what follows we consider the more general integro-differential operator
\begin{equation}\label{dif_op}
\mathcal{P}_\Phi u = u_{tt} - \partial_j(g^{ij}(x)\partial_iu) + \langle A(x),u'\rangle + b(x)u +\Phi*u,
\end{equation}
where $u'=(u_x,u_t)$, $g$ is a Riemannian metric, and the vector-function $A$, the scalar-function $b$ and the kernel $\Phi$ are continuous functions. 

\begin{remark}
The next two lemmas also hold if the coefficients $A$ and $b$ are analytic functions in $t$.
\end{remark}

\begin{lemma}\label{lemma1}
Let $\Omega$ and $T$ be as in Theorem \ref{uniqueness}. Let $f\in L^2(\Omega)$ and $u\in H^2(Q)$ be a solution of
\begin{equation}\label{eq_uniq}
\left\{\begin{array}{rl} \mathcal{P}_\Phi u=0&\text{in } (0,T)\times\Omega,\\ u|_{t=0}=0&\text{in }\Omega,\\ \partial_tu|_{t=0}=f&\text{in }\Omega.\end{array}\right.
\end{equation}
If  $u=\partial_\nu u = 0$ on $\partial Q(0)\cap\partial\Omega$, then 
$$u = 0 \text{  in  } Q(0),\text{ and in particular } f= 0\text{  in  }\Omega(0).$$
\end{lemma}


\begin{proof}
Given a point $y=(y',0)\in\partial\Omega$, let's consider local coordinates $(U,(x^1,...,x^{n-1}))$ in the boundary near $y'$. For $\epsilon\geq 0$ we define the sets
\begin{eqnarray}\label{set2}
Q_y(\epsilon) &=& \{(t,x)\in Q: \varphi(x,t)>\epsilon,\; x'\in U\},\\
\Omega_y(\epsilon) &=& \{x\in\Omega:(R-x^n)^2>r^2+\epsilon,\; x'\in U\}. \label{set2}
\end{eqnarray}
In what follows we take $r=R-\delta$, for some $\delta>0$ small enough, therefore $x^n\in[0,\delta)$ in the set $Q(0)$.

Let's first consider an arbitrary function $u\in C^\infty(\overline{Q}_y(0))$ such that $u=\partial_\nu u=0$ on $\overline{Q}_y(0)\cap\partial\Omega$, and let $\tilde{u}(t,x) = \chi(x') u(t,x)$, with $\chi\in C^\infty_0(U)$.
The idea is to obtain a well known local Carleman estimate for $\tilde{u}$ and later use it, along with a partition of unity, to get an analogous estimate in $Q(0)$.

Let's denote $\mathcal{P} = u_{tt} - \partial_j(g^{ij}(x)\partial_iu)$, the principal part of $\mathcal{P}_\Phi$. By analyzing the conjugate operator $\mathcal{P}_\tau = e^{\tau\varphi}\mathcal{P}e^{-\tau\varphi}$, it is possible to deduce (after long computations) a pointwise estimate for $v=e^{\tau\varphi}\tilde{u}$ of the form:
\begin{equation}\label{ptwise est}
\begin{aligned}
C|\mathcal{P}_\tau v|^2&\geq \tau(|v_t|^2+|v_n|^2) + \tau^3|v|^2+\text{div}_x(Y) + \partial_tZ\\
&\quad +4\tau(R-\delta)^2\left(\frac{1}{2}\partial_n\tilde{g}^{kl}v_kv_l\right) - 2\tau\gamma|v_x|^2_{\tilde{g}}
\end{aligned}
\end{equation}
for some constant $\gamma>0$ depending on the parameter $\alpha$ which is chosen small enough, and with $(Y,Z)$ a vector-valued function depending on lower order derivatives of $v$ and vanishing in $\partial Q_y(0)\backslash\{\varphi=0\}$. In fact, the previous follows by decomposing $\mathcal{P}_\tau v$ as the sum of two operators,
$$\mathcal{P}_+v = v_{tt} - \partial_j(g^{ij}\partial_iv) + \tau^2\Phi v,\quad \Phi = \varphi^2_t - |\varphi_x|^2_g$$
and
$$\mathcal{P}_-v = 2\tau\big(\langle \varphi_x,v_x\rangle_g - \varphi_t v_t\big) + \tau\Psi v,\quad \Psi = \partial_j(g^{ij}\partial_i\varphi) -\varphi_{tt},$$
and bounding from below the inequality
$$|\mathcal{P}_\tau v|^2\geq |\mathcal{P}_+v|^2 + 2(\mathcal{P}_+v)(\mathcal{P}_-v).$$
Here we apply the convexity condition on the level surfaces $\{x^n=s\}$ in \eqref{convexity_covec}. By choosing then $R$ large enough and $\delta$ small,
we arrive to the estimate
$$\tau^3|v|^2 + \tau(|v_t|^2 + |v_x|^2_g)\leq C\big(e^{\tau\varphi}|\mathcal{P}\tilde{u}|^2 - \text{div}_x(Y) - \partial_tZ\big).$$
Because $v=e^{\tau\varphi}\tilde{u}$, we can bound from below the left hand side of the previous inequality by similar terms but involving now the function $u$ (and the exponential weight function). 
Integration over $Q_y(0)$
and the Gauss-Ostrogradski\u{\i} formula give us that
\begin{equation}\label{car1}
\begin{aligned}
&\tau\int_{Q_y(0)}e^{2\tau\varphi}\big(\tau^2|\tilde{u}|^2+ |\tilde{u}_t|^2 + |\tilde{u}_x|^2_g\big)dxdt\\
&\leq C\int_{Q_y(0)}e^{2\tau\varphi}|\mathcal{P}\tilde{u}|^2dxdt+ C\int_{\Gamma_y(0)}\big(\langle X_1u',u'\rangle + \langle X_2,u'\rangle u + X_3|u|^2\big)dS, 
\end{aligned}
\end{equation}
where $dS$ denotes the surface measure on $\Gamma_y(0) = Q_y(0)\cap\{\varphi=0\}$, and the matrix-function $X_1(x,t)$, the vector-function $X_2(x,t)$, and the scalar-function $X_3(x,t)$ are some continuous functions depending on $\varphi$ and $Q_y(0)$. Using the continuity of the coefficients in the lower order terms (l.o.t) of $\mathcal{P}_\Phi$ and noticing that 
$$|\mathcal{P}\tilde{u}|^2 \leq 2|\mathcal{P}_\Phi \tilde{u}|^2 + 2|(\text{l.o.t of }\mathcal{P}_\Phi)\tilde{u}|^2,$$
we can choose $\tau_0$ larger if necessary and absorb the second summand in the right hand side above with the left hand side of \eqref{car1}. Then
\begin{equation}\label{car2}
\begin{aligned}
&\tau\int_{Q_y(0)}e^{2\tau\varphi}\big(\tau^2|\tilde{u}|^2+ |\tilde{u}_t|^2 + |\tilde{u}_x|^2_g\big)dxdt\\
&\leq C\int_{Q_y(0)}e^{2\tau\varphi}|\mathcal{P}_\Phi\tilde{u}|^2dxdt+ C\int_{\Gamma_y(0)}\big(\langle X_1u',u'\rangle + \langle X_2,u'\rangle u + X_3|u|^2\big)dS, 
\end{aligned}
\end{equation}

The analogous inequality in the larger set $Q(0)$ is obtained by considering a partition of unity and using the compactness of $\partial\Omega$. More precisely, let now $u\in C^\infty(\overline{Q}(0))$ and let $\{U_i\}_i$ be a finite covering of the boundary such that on each $U_i$ we can define boundary local coordinates, and let $\{\chi_i\}_i$ be a finite smooth partition of unity subordinate to $\{U_i\}_i$. We also consider a collection of points $y_i\in U_i$. Then, denoting $u_i=\chi_i^{1/2} u$, and the measure $d\sigma = dt d\text{Vol}(x)$ on $Q$, from the previous estimates we get
$$
\begin{aligned}
&\tau\int_{Q(0)}e^{2\tau\varphi}\big(\tau^2|u|^2 + |u_t|^2 + |u_x|^2_g\big)d\sigma\\
&\hspace{1em} =\tau\sum_i\int_{Q_{y_i}(0)}e^{2\tau\varphi}\chi_i\big(\tau^2|u|^2 + |u_t|^2 + |u_x|^2_g\big)dxdt \\
&\hspace{1em} \leq \tau\sum_i\int_{Q_{y_i}(0)}e^{2\tau\varphi}\big(\tau^2|u_i|^2 + |(u_i)_t|^2 + |(u_i)_x|^2_g\big)dxdt\\
&\hspace{5em}  + C\tau\int_{Q(0)}e^{2\tau\varphi}|u|^2d\sigma\\
&\hspace{1em} \leq C\left(\int_{Q(0)}e^{2\tau\varphi}|\mathcal{P}_\Phi u|^2d\sigma +\int_{\Gamma(0)}\big(\langle X_1u',u'\rangle + \langle X_2,u'\rangle u + X_3|u|^2\big)dS\right.\\
&\hspace{5em}\left.+ \sum_i\int_{Q(0)}e^{2\tau\varphi}|[\mathcal{P}_\Phi,\chi_i]u|^2dxdt + \tau\int_{Q(0)}e^{2\tau\varphi}|u|^2d\sigma\right),\\
\end{aligned}
$$
where notice $[\mathcal{P}_\Phi,\chi_i]$ are differential operators of order 1. We absorb the interior integrals with lower order derivatives of $u$ using the left hand side and get
\begin{equation}\label{car3}
\begin{aligned}
&\tau\int_{Q(0)}e^{2\tau\varphi}\big(\tau^2|u|^2 + |u_t|^2 + |u_x|^2_g\big)d\sigma&\\
&\hspace{1em} \leq C\int_{Q(0)}e^{2\tau\varphi}|\mathcal{P}_\Phi u|^2d\sigma+ C\int_{\Gamma(0)}\big(\langle X_1u',u'\rangle + \langle X_2,u'\rangle u + X_3|u|^2\big)dS.
\end{aligned}
\end{equation}
It follows from a density argument  that the previous estimate also holds for functions in $H^2(Q)$ with null Cauchy data in $\partial Q(0)\cap\partial\Omega$.

Let $u$ be as in the hypothesis of the lemma. Then, $u$ satisfies an inequality of the form \eqref{car3}, without the interior integral in the right hand side. Noticing that the boundary integral does not depends on $\tau$, we let $\tau$ goes to infinity and conclude that $u=0$ in $Q(0)$.
\end{proof}


The aim of the second lemma is to extend the time for which $u$ is zero. Based again on Carleman estimates we will be able to succeed until we hit the characteristic surface associated to the principal part of $\mathcal{P}_\Phi$, this is the surface $\{(t,x): T-t=\text{dist}(x,\partial\Omega)\}$. 

\begin{lemma}\label{lemma2}
Let $\Omega$ and $T$ be as in Theorem \ref{uniqueness}. If $u\in H^2(Q)$ is a solution of \eqref{eq_uniq}, then 
$$u=0\quad\text{in}\quad\{(t,x)\in Q:\text{\rm dist}(x,\partial\Omega)<\epsilon,\;0<t<T-\text{\rm dist}(x,\partial\Omega)\}$$
for some $0<\epsilon\leq T$.

\end{lemma}

\begin{proof}
From Lemma \ref{lemma1}, $u=0$ in some neighborhood $\{(t,x)\in Q: (R-x^n)^2>\alpha t^2 + r^2\}$ for appropriate constants $\alpha,R,r$. It is clear that for sufficiently small $\epsilon_1,\epsilon_2>0$, the previous set contains $[0,\epsilon_1]\times\{x\in\Omega:\text{dist}(x,\partial\Omega)<\epsilon_2\}$.

In a neighborhood of $\partial \Omega$ we define
\begin{equation}\label{psi}
\psi(t,x) := (\epsilon_2-x^n)(T-t-x^n),
\end{equation}
and for $\gamma>0$ we consider the sets
$$Q_{\gamma}^{\epsilon_2} := \{(t,x)\in Q\;|\; \psi(t,x)>\gamma,\; x^n<\epsilon_2\},$$
which exhaust $Q^{\epsilon_2} = \{(t,x)\in Q\;|\; x^n<\epsilon_2,\;0<t< T-x^n\}$, this is $Q^{\epsilon_2} = \bigcup_{\gamma>0}Q_{\gamma}^{\epsilon_2}$. Moreover, there exists $\gamma_0>0$ such that 
$$\emptyset\neq Q_{\gamma_0}^{\epsilon_2} \subset [0,\epsilon_1]\times\{x\in \Omega : x^n<\epsilon_2\}.$$

We denote by $B(t_0,x_0;r)$ the ball centered at $(t_0,x_0)$ and radius $r$ for the euclidean metric. Given the following 
\vspace{.5em}

\noindent{\bf Claim.} {\it Suppose that for $(t_0,x_0)\in Q^{\epsilon_2}$, $u$ vanishes below the level surface $\{\psi(x,t) = \psi(t_0,x_0)\}$ near $(t_0,x_0)$, this is in $Q_{\psi(t_0,x_0)}^{\epsilon_2}\cap B(t_0,x_0;r)$ for some $r>0$. Then, $u=0$ in a neighborhood of $(x_0,t_0)$. }
\vspace{.5em}

\noindent the proof of the lemma is complete by the next argument. Let's assume that $\text{supp}u\cap Q^{\epsilon_2}\neq \emptyset$. We can find $0<\gamma^*\leq\gamma_0$ such that
$$\text{supp}u\cap Q_{\gamma}^{\epsilon_2} = \emptyset,\;\forall \gamma>\gamma^*\quad\text{and}\quad \text{supp}u\cap \{(t,x)\in Q^{\epsilon_2} : \psi(t,x)=\gamma^*\}\neq \emptyset.$$
The application of the claim on every contact point $(t^*,x^*)\in \text{supp}u\cap \{(t,x)\in Q^{\epsilon_2} : \psi(t,x)=\gamma^*\}$, contradicts the choice of $\gamma^*$. Consequently, we deduce that $u=0$ on every $Q_{\gamma}^{\epsilon_2}$, $\gamma>0$, and therefore $u = 0$ in $Q^{\epsilon_2}$.

It only remains to show the previous claim. Here is where Carleman estimates play a fundamental role, and as before we will consider a particular choice of weight function which needs to fulfill a pseudo-convex condition with respect to $\mathcal{P} = \partial^2_t - \partial_{x^j}\big(g^{ij}\partial_{x^i}\cdot\big)$, in the set $\{(0,\xi)\in T^*_{(t_0,x_0)} \Omega\}$. Moreover, we will take it to be linear and non-increasing in time. Provided the above, it is possible to apply a pseudo-differential Carleman estimate introduced in \cite{Tataru1} and conclude that $u$ vanishes near $(t_0,x_0)$.

Let's consider local coordinates in $\partial\Omega$ near some $y\in\partial\Omega$ such that in those coordinates $y = (x_0',0)$. For some $\delta>0$ to be appropriately chosen, we define the following weight function
$$\varphi(t,x) = \psi(t,x)-\psi(t_0,x_0) - \frac{1}{2}\delta|x-x_0|^2$$
where here $|\cdot|$ stands for the euclidean norm and $\psi$ as in \eqref{psi}. 
Denoting the principal symbol of $\mathcal{P}$ by $p(t,x;\theta,\xi) = -\theta^2 + |\xi|^2_g$, where $|\xi|^2_g = g^{ij}(x)\xi_i\xi_j$ is the norm on covectors induced by $g$, the pseudo-convexity condition requires to show that $\varphi$ satisfies
\begin{enumerate}
\item[(1)] $\text{Re}\{\bar{p},\{p,\varphi\}\}(t_0,x_0;0,\xi)>0 \text{ for all } \xi\neq 0 \text{ such that } p(t_0,x_0;0,\xi) = 0,$
\item[(2)] $\frac{1}{i\tau}\{\bar{p}_\varphi,p_\varphi\}(t_0,x_0;0,\xi;\tau)>0\text{ for all } \xi\neq 0,\;\tau>0\\\text{ such that }p_\varphi(t_0,x_0;0,\xi,\tau)=0.$
\end{enumerate}Here $p_\varphi(t_0,x_0;0,\xi;\tau) = p(t,x; \theta + \text{i}\tau\varphi_t, \xi + \text{i}\tau\varphi_x)$ and $\{\cdot,\cdot\}$ is the Poisson bracket 
$$\{f,h\} = \sum_{j=1}^n\frac{\partial f}{\partial \xi_j}\frac{\partial h}{\partial x_j} - \frac{\partial f}{\partial x_j}\frac{\partial h}{\partial \xi_j} + \frac{\partial f}{\partial \theta}\frac{\partial h}{\partial t} - \frac{\partial f}{\partial t}\frac{\partial h}{\partial \theta}.$$
Recall that we are working in boundary normal coordinates hence the metric $g$ takes the form \eqref{metric_nbc}. 
The first condition is trivially fulfilled since the principal symbol $p$ is elliptic in the set $\{\theta = 0\}$. Let's use the following notation: the variable appearing in the subindex means we are differentiating with respect to such variable, for instance $\varphi_{x'} = \partial_{x'}\varphi$ and $\varphi_{tx^n} = \partial_{t}\partial_{x^n}\varphi$. To verify the second condition we notice first that $\varphi_{x'}(t_0,x_0) = \psi_{x'}(t_0,x_0) = 0$, $\varphi_{x^n}(t_0,x_0) = \psi_{x^n}(t_0,x_0)= -\alpha$ and $\varphi_t(t_0,x_0) = \psi_t(t_0,x_0) = -\beta$ where $\alpha>\beta>0$. In fact,
$$\alpha = (\epsilon_2-x^n_0) + (T-t_0-x^n_0),\quad\text{and}\quad \beta = \epsilon_2-x^n_0.$$
Also, denoting $\delta_{ij}$ the Kronecker delta,
$$\varphi_{tt} = 0,\quad \varphi_{tx^i} = \delta_{in},\quad\varphi_{x^ix^j} = 2\delta_{in}\delta_{jn} - \delta\cdot\delta_{ij} .$$
Secondly, it is easy to check that $p_\varphi(t_0,x_0;0,\xi,\tau)=0$ is equivalent to $\xi_n = 0$ and $|\xi'|^2_{\tilde{g}} = \tau^2(\alpha^2-\beta^2)$.
Then, after some tedious computations, in the set of points $(t_0,x_0;0,\xi;\tau)$ such that $p_\varphi = 0$, we get
\begin{align*}
\frac{1}{i\tau}\{\bar{p}_\varphi ,p_\varphi\} &= \frac{1}{\tau}\{\text{Re}p_\varphi,\text{Im}p_\varphi\}\\
&=8\tau^2(\alpha^2 - \alpha\beta)  + 4\alpha\left(\frac{1}{2}\partial_n\tilde{g}^{ij}\right)\xi'_i\xi'_j - \delta M,
\end{align*}
with $M = 4\tau^2\alpha^2 + 4\left(\tilde{g}^{jk}\xi'_j\right)\left(\tilde{g}^{ik}\xi'_i\right)$ such that, for some $C>0$,
$$M\leq4\tau^2(\alpha^2 + C(\alpha^2-\beta^2)).$$
Let's recall the positive-definiteness of the second fundamental form in \eqref{convexity_covec}, and denote $\kappa = \min_{s\in[0,\epsilon_2]}\kappa_{s}$.
By choosing $\delta>0$ small enough we obtain that
$$\frac{1}{i\tau}\{\bar{p}_\varphi ,p_\varphi\} \geq 8\tau^2\alpha(\alpha - \beta)\left(1+\frac{\kappa}{2}(\alpha + \beta)\right) - 4\delta\tau^2(\alpha^2 + C(\alpha^2-\beta^2))>0,$$
therefore $\varphi$ satisfies the second condition of pseudo-convexity. It follows from \cite[Theorem 3]{Tataru2} that there exists $\eta,C,d>0$ such that any function $v$ supported inside $B(t_0,x_0;\eta)$ (we of course choose $0<\eta<r$), for which the RHS of the next inequality is finite, satisfies the pseudo-differential Carleman estimate
\begin{equation}\label{tataru_carleman}
\tau^{-1}\|Ee^{\tau\varphi}v\|^2_{(2,\tau)}\leq C\left(\|Ee^{\tau\varphi}P v\|^2 + e^{-d\epsilon\tau}\|e^{\tau\varphi}P v\|^2 + e^{-d\epsilon\tau}\|e^{\tau\varphi}v\|^2_{(1,\tau)}\right),
\end{equation}
for the weighted norms
$$\|v\|^2_{(m,\tau)} := \sum_{|\alpha|+j\leq m}\tau^{2(m-|\alpha|-j)}\|D^\alpha D^j_tv\|^2_{L^2(\RR^{n+1})},\quad\tau>0;\quad\|\cdot\|:=\|\cdot\|_{(0,\tau)},$$
and the pseudo-differential operator $E :=e^{\frac{\epsilon}{2\tau}|D_t|^2}$. This operator can also be considered as the convolution operator
$$Ev(x,t) = \left(\frac{\tau}{2\pi\epsilon}\right)^{1/2}\int e^{-\frac{\tau|t-s|^2}{2\epsilon}}v(x,s)ds.$$
We would like to apply the above Carleman estimate to $u$ and eventually deduce that $u$ vanishes near $(t_0,x_0)$. With that in mind we need first to localize it near $(t_0,x_0)$. As in \cite{2007Bukhgeim}, in $(\psi'(t_0,x_0))^{\perp} = \{(\theta,\xi):\langle \psi'(t_0,x_0),(\theta,\xi)\rangle_{e\otimes g} = 0\}$ we see that $|\theta|\leq C_1|\xi|_g$, hence
$$\langle (\psi-\varphi)''(\theta,\xi),(\theta,\xi)\rangle_g = \delta|\xi|^2_g\geq c_2|(\theta,\xi)|^2_{e\otimes g}.$$
Therefore, by choosing $l_1<0$ small enough in magnitude, the set $\{\varphi(t,x)> l_1\}\cap\{\psi(t,x)<\psi(t_0,x_0)\}$ is contained in a sufficiently small vicinity of $(t_0,x_0)$. We then localize $u$ by multiplying it with a function of the form $\chi(\varphi(t,x))$ with $\chi\in C^\infty(\RR)$ a nondecreasing function such that
$$\chi(s) = \left\{\begin{matrix}0&\text{for }s< l_1,\\ 1&\text{for }s>l_2 ,\end{matrix}  \right.$$
where $l_1<l_2<0$ are small enough in magnitude, then
$$\text{supp}\big[u(t,x)\chi(\varphi(t,x))\big]\subset B(t_0,x_0;\eta).$$ 
In what follows we write $\chi$ meaning the composition $\chi\circ\varphi$.
Consequently, $v = \chi u$ satisfies the inequality \eqref{tataru_carleman}.
We include the integral term in the estimates by noticing that 
$$\mathcal{P}(\chi u) = \chi \mathcal{P}u + [\mathcal{P},\chi]u = \chi \mathcal{P}_\Phi u - \chi \Phi*u + \mathcal{P}_1u,$$
where $\mathcal{P}_1$ is a differential operator of order 1 with coefficients supported in $\{(t,x)|\varphi(t,x)<l_2\}$. Consequently
\begin{equation}\label{tataru_carleman2}
\begin{aligned}
\tau^{-1}\|Ee^{\tau\varphi}(\chi u)\|^2_{(2,\tau)}\leq c\Big(\|E&e^{\tau\varphi}\mathcal{P}_1 (\chi u)\|^2 + \|Ee^{\tau\varphi}\chi(\Phi*u)\|^2\\
&+\left. e^{-d\epsilon\tau}\|e^{\tau\varphi}\mathcal{P} (\chi u)\|^2 + e^{-d\epsilon\tau}\|e^{\tau\varphi}(\chi u)\|^2_{(1,\tau)}\right).
\end{aligned}
\end{equation}
The idea in what remains of the proof is to estimate $\|Ee^{\tau\varphi}(\chi u)\|_{(2,\tau)}$ by a term of the form $e^{l\tau}$, with $l<0$, and use \cite[Proposition 4.1]{Tataru1} to conclude that $\chi u=0$ in $\{(t,x)|\varphi(t,x)>l\}$. Such estimate is obtained in exactly the same way as in the proof of Lemma 6 in \cite{2007Bukhgeim}, where everything reduces to estimate the term with the convolution since the other terms in the right hand side of the last inequality are easily bounded. For the arguments needed to conclude the claim we refer the reader to \cite{2007Bukhgeim}.
\end{proof}

\noindent{\it Proof of Theorem \ref{uniqueness}.}
Let $u$ be a solution of $P_\Phi = 0$ with initial conditions $[f,-af]$ and such that $\Lambda_\Phi u = 0$. Due to our assumption on the coefficients of $P_\Phi$, $u$ solves $(\partial^2_t-\Delta)u=0$ in $(0,T)\times(\RR^n\setminus\overline{\Omega})$ with null initial and Dirichlet boundary data. Then, for any $x_0\in \RR^n\setminus\overline{\Omega}$, $u$ vanishes in $(0,T)\times V$ for some small neighborhood $V$ of $x_0$ such that $\overline{V}\cap\overline{\Omega} = \emptyset$. The previous is a consequence of a sharp domain of dependence for the wave operator in the exterior problem (see \cite[Proposition 2]{Finch_Rakesh_2004}). Then $u=0$ in $(0,T)\times(\RR^n\setminus\Omega)$ which implies null Neumann data, $\frac{\partial u}{\partial\nu}\big|_{(0,T)\times\partial\Omega}=0$. 

Let's set 
\begin{equation}\label{u_bar}
\bar{u}(t,x) = \int^t_0u(s,x)ds\quad\text{and}\quad \Psi(t,x) = -\int^\infty_t\Phi(s,x)ds.
\end{equation}
Note that $\bar{u}_t(t,x) = u(t,x)$ and $\partial_t\Psi = \Phi$.
Moreover
$$\partial_t\Big( \int^t_0\Psi(t-s,x)\bar{u}_s(s,x)ds\Big) = \Psi(0,x)\bar{u}_t(t,x) + \int^t_0\Phi(t-s,x)\bar{u}_s(s,x)ds,$$
which, since $\bar{u}(0,x) = 0$ and integration by parts, implies
\begin{equation}
\begin{aligned}\label{Phi_Psi}
\int^t_0\Big( \int^\tau_0\Phi&(\tau-s,x)u(s,x)ds\Big) 
&=  \int^t_0\Phi(t-s,x)\bar{u}(s,x)ds.
\end{aligned}
\end{equation}
We integrate equation \eqref{TAT2} on the interval $(0,t)$ for any $t>0$. It follows from the previous computations that $\bar{u}$ solves a system of the form \eqref{eq_uniq} 
with vanishing Cauchy data. In addition, notice that $\bar{u}_{tt} = u_t\in L^2(Q)$, so using equation \eqref{eq_uniq} we get $c^{2}\Delta \bar{u}\in L^2(Q)$, which by elliptic regularity implies $\bar{u}\in H^2(Q)$. 
We can now apply Lemma \ref{lemma2} on $\bar{u}$ and conclude that $u = 0$ in a set of the form $\{(t,x)\in Q: x^n<\epsilon\; 0<t<T-x^n\}$. This implies we have reduced the problem to the smaller domain $[0,T-\epsilon]\times\{x\in \Omega: x^n>\epsilon\}$. If $\epsilon=T$ we are done, otherwise we can apply again Lemma \ref{lemma2} in the new domain. Iterating this process we conclude the result. \hfill $\blacksquare$\\

There is a common condition appearing in the literature of Carleman estimates and inverse problems related to the wave equation with variable sound speed. It assumes the existence of some $x_0\in\RR^n$ for which
\begin{equation}\label{cond_c}
(x-x_0)\cdot\partial_xc(x)<c(x) \quad\forall x\in\RR^n.
\end{equation}
In geometric terms, \eqref{cond_c} says that the spheres with center at $x_0$ are strictly convex for the metric $c^{-2}dx^2$ \cite[\S 3]{2013StefanovUhlmann}. 
Such collection of spheres can then be used to foliate the domain $\Omega$ and, as you will see in the next corollary (see also Figure \ref{fig:foliation}), it allows us to prove unique continuation and consequently uniqueness for the inverse problem without the assumption of $\Omega$ and the level surfaces of the distance function, $\text{dist}(\cdot,\partial\Omega)$, being strictly convex. The price we pay by removing the convexity requirement on $\Omega$ is the lost of sharpness in the bound of $T$ that guarantee uniqueness. 

\begin{figure}
\centering
\includegraphics[width=1\textwidth]{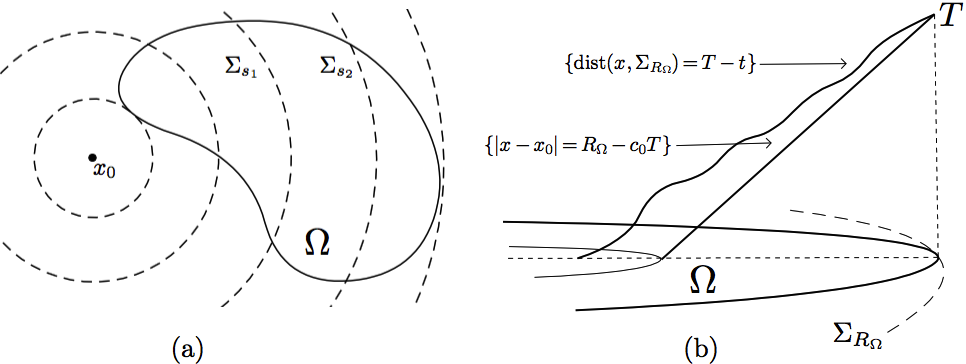}
\vspace{-.7em}
\caption{(a) Foliation of $\Omega$ by Euclidean spheres $\{\Sigma_s\}_s$ centered at $x_0$. (b) Sub-characteristic unique continuation under condition \eqref{cond_c}.}
\label{fig:foliation}
\end{figure}

Let $\Omega\subset\RR^n$ be an open and bounded subset with $\partial\Omega$ smooth, and $T>0$. We  assume the sound speed $c(x)$ satisfies condition \eqref{cond_c} and assume the constant $c_{0}>0$ is a lower bound for the sound speed. Let's denote 
$$R_{\Omega}=\max_{r>0}\{|x-x_0|:x\in\partial\Omega\},$$
$$r_{\Omega} = \left\{\begin{array}{ll} \displaystyle\min_{r>0}\{|x-x_0|:x\in\partial\Omega\},&\text{if }x\in\RR^n\backslash\bar{\Omega}\\
0,&\text{otherwise},\end{array}\right.$$
and $D_\Omega = R_\Omega-r_\Omega.$
\begin{theorem}\label{uniqueness2}
Assume $\Omega$, $T$ and $c$ are as above, and as in the TAT problem, we assume $P_\Phi = \partial^2_t - \Delta$ outside $\Omega$. If $u\in H^2(Q)$ is a solution of \eqref{eq_uniq} such that $u=\partial_\nu u=0$ on $(0,T)\times\partial\Omega$, then
$$u=0\quad\text{in}\quad \{(t,x)\in Q: 0<t<T-c^{-1}_0(R_{\Omega}-|x-x_0|)\}.$$
As a consequence, in the thermoacoustic problem, if $f\in H_{D}(\Omega)$ is such that $\Lambda_\Phi \ff=0$, with $\ff = [f,-af]$, then 
$$f=0\quad\text{in}\quad \{x\in\Omega: |x-x_0|>R_\Omega-c_{0}T\},$$
and in particular, $f\equiv 0$ when $T\geq c^{-1}_0D_\Omega.$
\end{theorem}

\begin{remark}
From \cite[Proposition 7.1]{UhlmannStefanov2012InsideOut}, the uniqueness time defined in Theorem \ref{uniqueness} satisfies $T_0< c^{-1}_0D_\Omega$.
\end{remark}

\begin{proof}
Let's extend $u$ to be zero outside $\Omega$ in the interval $[0,T]$. Due to the null Cauchy data, finite speed of propagation and the well-posedness of the exterior problem, $u$ solves \eqref{eq_uniq} in the whole space. Notice that in particular, $u=\partial_\nu u=0$ on the Euclidean sphere $\{x\in\RR^b:|x-x_0|=R_\Omega\}$, for all $t\in[0,T]$. $\partial_\nu$ stands for a generic exterior normal derivative.

We denote $\Sigma_r = \{x\in\RR^b:|x-x_0|=r\}$ the sphere of center $x_0$ and radius $r$, and we set $r_0=\max\{0,c_{0}T-D_\Omega\}$. By hypothesis, $\Sigma_r$ with $r\in[r_0,R_\Omega]$ are strictly convex surfaces for the metric $c^{-2}dx^2$ that foliate $\Omega$ (see Figure \ref{fig:foliation}(a)). For a given $r\in[r_0,R_\Omega]$, let's assume that 
$$u=\partial_\nu u=0\quad\text{on}\quad [0,T-c^{-1}_0(R_\Omega -r)]\times\Sigma_r.$$
Since $\Sigma_r$ is strictly convex we can apply Lemma \ref{lemma2} with $\Omega$ replaced by $B(x_0,r)$, the Euclidean ball of center $x_0$ and radius $r$, and deduce that $u=0$ in
$$\{(t,x)\in (0,T)\times B(x_0,r): \text{dist}(x,\Sigma_r)<\epsilon,\; t<T-c^{-1}_0(R_\Omega-r)-\text{dist}(x,\Sigma_r)\},$$
for some $\epsilon>0$. Recalling that $c_{0}$ is a lower bound for $c$, we have that 
$$\text{dist}(x,\Sigma_r)<c^{-1}_0(r-|x-x_0|)\quad\forall x\in B(x_0,r),$$ 
therefore we can find $r_1\in(0,r)$ such that $u$ vanishes in the smaller set
$$\{(t,x): r_1<|x-x_0|< r,\; 0<t<T-c^{-1}_0(R_\Omega-|x-x_0|)\}.$$
Moreover, $u$ has null Cauchy data on $\Sigma_{r_1}$ for all $t\in(0,T-c^{-1}_0(R_\Omega-r_1))$ (see Figure \ref{fig:foliation}(b)).

If we denote by $s$ the infimum of the radius $r\geq r_0$ for which $u$ has vanishing Cauchy data in $(0,T-c^{-1}_0(R_\Omega - r))\times\Sigma_{r}$, by the first paragraph and the previous argument we know $s< R_\Omega$ (since $T>0$). Moreover, if $s>r_0$, it must also satisfies the same property, this is, $u=\partial_\nu u=0$ in $\Sigma_s$ for all $t\in(0,T-c^{-1}_0(R_\Omega - s))$. Consequently, we can still apply the arguments in the paragraph above which leads us to conclude $s=r_0$.

Let now $f\in\mathcal{H}_{D,a}(\Omega)$ be as in the hypothesis, and $u$ solution of \eqref{TAT2}. Analogously to the proof of Theorem \ref{uniqueness}, the function $\bar{u}$ defined in \eqref{u_bar} satisfies a system of the form \eqref{eq_uniq} with null Cauchy data. The result then follows directly from the previous.
\end{proof}
\section{Stability}
The stability with complete data follows directly from the analogous results for the damped and undamped case. Due to the microlocal nature of this property, the minimum time needed to recover $f$ in a stable way is usually larger than the uniqueness time. Indeed, it's necessary to capture information coming from every singularity of the initial source. In a non-trapping domain, such lower bound is related to the value
$$T_1(\Omega) = \sup\{|\gamma|_{g}: \gamma \subset \bar{\Omega} \text{ geodesic for the metric } g = c^{-2}dx^2\},$$
being $\frac{1}{2}T_1$ when there is no damping coefficient and exactly $T_1$ for the damped case. Notice that $T_1>2T_0$ and in the case $c$ satisfies \eqref{cond_c}, $T_1/2\leq (R_\Omega-r_\Omega)/(\alpha c_{0})$ with (see \cite[Proposition 7.1]{UhlmannStefanov2012InsideOut})
\begin{equation}\label{alpha}
\alpha = \min_{x\in\bar{\Omega}}(1-c^{-1}(x-x_0)\cdot\partial_xc)>0.
\end{equation}

\begin{theorem}
Let $\Omega$ be strictly convex for the metric $g=c^{-2}dx^2$. Assume that $\Omega$ and $T$ are as in Theorem \ref{uniqueness}  (or as in Theorem \ref{uniqueness2}). In addition, assume $T_1(\Omega)<T<\infty$ if $a\neq 0$ and $\frac{1}{2}T_1(\Omega)<T<\infty$ otherwise (resp. $2\alpha^{-1}c^{-1}_0D_\Omega<T<\infty$ and $\alpha^{-1}c^{-1}_0D_\Omega<T<\infty$). Then there exists $C>0$ such that
$$\|f\|_{H_{D}(\Omega)}\leq C\|\Lambda_\Phi f\|_{H^1((0,T)\times\partial\Omega)}.$$
\end{theorem}
\begin{proof}
The idea is to compare the observation operator $\Lambda_\Phi$ with its analogous for the undamped and damped case, $\Lambda_0$ and $\Lambda_a$ respectively. These last two operators are known to be stable maps (see \cite{TATbrain} and \cite{Homan} respectively) and furthermore, from the results of the previous section, we know $\Lambda_\Phi$ is injective. The proof then reduces to show that the respective error operators are compact. We only show this for the case $a\equiv 0$, the proof when there is a damping coefficient is obtained analogously.

From \cite{TATbrain} follows there is a constant $C>0$ such that
$$\|f\|_{H_{D}}\leq C\|\Lambda_0f\|_{H^1}\leq C\|\Lambda_\Phi f\|_{H^1} + C\|(\Lambda_\Phi - \Lambda_0)f\|_{H^1}.$$
Let's denote $R = \Lambda_\Phi - \Lambda_0$ and $u$ the attenuated wave related with $\Lambda_\Phi$. Then, $R$ maps $f\in H_D(\Omega)$ to the boundary data $w|_{(0,T)\times\partial\Omega}$ of the system
\begin{equation}
\left\{\begin{array}{ll} (\partial^2_t-c^{2}\Delta+b)w = -\Phi*u,&(t,x)\in(0,T)\times\RR^{n}\\ w|_{t=0}=0,\\ w_t|_{t=0}=0.\end{array}\right.
\end{equation}
By finite propagation speed we can work in a larger domain $\Omega'$ such that $w =u= 0$ on its boundary and outside $\Omega'$. Due to the higher regularity theorem in \cite[\S 7.2.3 Theorem 5]{evans1998partial}, since $F(t,x) = -[\Phi*u](t,x)$ satisfies $F,F_t\in L^2((0,T);L^2(\Omega'))$, we obtain that $w\in C((0,T);H^2(\Omega'))$ and $w_t\in C((0,T);H^1(\Omega'))$, and consequently the trace of $w$ in $\partial\Omega$ belongs to $H^{3/2}((0,T)\times\partial\Omega)$, with the latter space compactly embedded in $H^1((0,T)\times\partial\Omega)$. 

The stability inequality is obtained by recalling the injectivity of $\Lambda_\Phi$ from Theorem \ref{uniqueness} (respectively Theorem \ref{uniqueness2}) and applying the classical result \cite[Proposition V.3.1]{Taylor}. 
\end{proof}
\section{Reconstruction}

We aim to construct a Neumann series that allow us to recover $f$ in \eqref{TAT2} from boundary measurements as it has been done in \cite{TAT,TATbrain,2015AcostaMontalto,2015StefanovYang,2015Nguyen} for the unattenuated case, and in \cite{Homan,2016Palacios} for the damped wave equation. However, due to the convolution term we need to modified the equation satisfied by the time reversed wave. Considering the same equation in the backward direction would imply the knowledge of the future. The strategy then is to solve a time reversal problem in such a way that the initial energy of the error function is bounded by the total energy (kinetic, potential and energy lost by attenuation) of the forward wave, inside the domain and at time $T$, analogously as the argument presented in \cite{2016Palacios}. Such total energy in the whole space has the attribute of being conserved in time, fact that allows us to reduce the proof to an estimate involving the norm of the initial source and the energy of the forward wave outside $\Omega$ (see Proposition \ref{prop}). The estimate says that at time $T$ a significant portion of the energy lies outside the domain. It was first used in \cite{TATbrain} and subsequently applied in \cite{2016Palacios}.

Let's introduce the following convolution-type operator
\begin{equation}\label{kernel Psi}
[\Phi\tilde{*}v](s,x) = \int^T_s\Phi(t-s,x)v(t,x)dt,
\end{equation}
which is the adjoint operator of $\Phi*(\cdot)$ under the $L^2$ inner product in $(0,T)$, this is, for any $L^2$-functions $u,v$, 
\begin{equation}\label{identity1}
\langle \Phi*u,v\rangle_{L^2(0,T)} = \langle u,\Phi\tilde{*}v\rangle_{L^2(0,T)}.
\end{equation}
Indeed, denoting by $\chi_I$ the indicator function in the interval $I\subset\RR$,
\begin{align*}
\int^{T}_{0}\left[\Phi*u\right](t)v(t)dt &= \int\int\chi(t)_{[0,T]}\chi(s)_{[0,t]}\Phi(t-s)u(s)v(t)dsdt\\
&= \int\int\chi(s)_{[0,T]}\chi(t)_{[s,T]}\Phi(t-s)u(s)v(t)dsdt \\
&= \int^{T}_{0}\big[\Phi\tilde{*}v](s)u(s)ds.
\end{align*}

Following the same approach than the latest results in reconstruction for TAT in the enclosure case as well as in the attenuated case for the damped wave equation, the idea is to consider the right back projection system that will make the error operator to be a contraction.
In the same way as in the proof of uniqueness, instead of working with $u$ with set
$$\bar{u}(t,x) = \int^t_0u(s,x)ds,$$
and $\Psi(t,x)$ as in \eqref{def:kernel2}.

Then, they satisfy 
\begin{equation}\label{eq_bar_u}
\left\{\begin{array}{rl} \partial_t^2\bar{u}- c^2\Delta\bar{u} +a\partial_t\bar{u} +p\bar{u}+ \Psi*\partial_t\bar{u}=0&\text{in } (0,T)\times\RR^n,\\ \bar{u}|_{t=0}=0&\text{in }\RR^n\\ \partial_t\bar{u}|_{t=0}=f&\text{in }\RR^n\end{array}\right.
\end{equation}
with $p(x) = b(x) - \Psi(x,0)\geq 0$. Notice we do not use \eqref{Phi_Psi} to obtain an equation as in \eqref{eq_uniq} and we keep a derivative inside the convolution. If $\bar{\Lambda}_\Psi:L^2(\Omega;c^{-2}dx)\to H^1((0,T)\times\partial\Omega)$ denotes the observation operator for this problem, this is $\bar{\Lambda}_\Psi f = \bar{u}|_{(0,T)\times\partial\Omega}$, by well-posedness of the direct problem we have the following relation, 
$$\bar{\Lambda}_\Psi f = \int^t_0[\Lambda_\Phi f](t)dt,\quad\forall f\in H_{D}(\Omega).$$
For the data $\bar{h} = \bar{\Lambda}_\Psi f$, we consider the solution $v$ of the system
\begin{equation}\label{TR}
\left\{\begin{array}{rcl}
(\partial^2_t-c^{2}\Delta -a\partial_t+ p - \Psi\tilde{*}\partial_t)v&=& 0\; \text{ in }(0,T)\times\Omega,\\
v|_{t=T}&=&\phi,\\
v_t|_{t=T}&=&0,\\
v|_{(0,T)\times\partial\Omega}&=&\bar{h},
\end{array}\right.
\end{equation}
with $\phi$ the harmonic extension of $\bar{h}(T,\cdot)$ in $\Omega$. Notice that problem \eqref{TR} is well-posed. This is due to the convolution term that involves values of $v$ in the interval $(t,T)$, thus by doing the change of variables $t\to T-t$ we get an IBVP of the form \eqref{eq_bar_u} which is uniquely solvable. We define the Time Reversal operator by 
$$A:H^1_{(0)}([0,T]\times\partial\Omega)\to L^2(\Omega;c^{-2}dx),\quad Ah = v_t(0,\cdot),$$
and denote by $K$ the error operator defined as follows,
$$K:L^2(\Omega;c^{-2}dx)\to L^2(\Omega;c^{-2}dx),\quad Kf = w_t(0,\cdot),$$
with $w = \bar{u}-v$, the error function that solves problem \eqref{w eq}.

In what follows we suppose the domain $\Omega$ is non-trapping (i.e. $T_0(\Omega)<\infty$). The main result of this section is the next
\begin{theorem}\label{teo:Nseries}
Let $\Omega$ be strictly convex for the metric $g=c^{-2}dx^2$. Assume that $\Omega$ and $T$ are as in Theorem \ref{uniqueness}  (or as in Theorem \ref{uniqueness2}). In addition, assume $T_1(\Omega)<T<\infty$ if $a\neq 0$ and $\frac{1}{2}T_1(\Omega)<T<\infty$ otherwise (resp. $2\alpha^{-1}c^{-1}_0D_\Omega<T<\infty$ and $\alpha^{-1}c^{-1}_0D_\Omega<T<\infty$, with $\alpha$ as in \eqref{alpha}). Then $A\bar{\Lambda}_\Psi = \text{Id}-K$ with $\|K\|_{\mathcal{L}(L^2(\Omega;c^{-2}dx))}<1$, and for any initial condition of \eqref{TAT2} of the form $\ff=(f,-af)$ with $f\in H_{D}(\Omega)$, the thermoacoustic inverse problem has a reconstruction formula given by
$$f=\sum^\infty_{m=0}K^mA\bar{h},\quad \bar{h} = \bar{\Lambda}_\Psi f.$$
\end{theorem}
\begin{proof}
Notice the error function $w = \bar{u}-v$ satisfies the equation
\begin{equation}\label{w eq}
\left\{\begin{array}{rcl}
(\partial_t^2 - c^{-2}\Delta + p) w&=& -a\bar{u}_t - av_t-\Psi*\partial_t\bar{u}- \Psi\tilde{*}\partial_t{v}\; \text{ in }(0,T)\times\Omega,\\
w|_{t=T}&=&\bar{u}^T - \phi,\\
w_t|_{t=T}&=&\bar{u}^T_t,\\
w|_{(0,T)\times\Gamma}&=&0.
\end{array}\right.
\end{equation}
with $\bar{u}^T = \bar{u}(T,\cdot)$ and $\partial_t\bar{u}^T = \partial_t\bar{u}(T,\cdot)$.
Moreover, we can write 
$$Kf = f - A\bar{h} = w_t(0),\quad \text{with}\quad\bar{h} = \bar{\Lambda}_\Psi f.$$
We want to estimate the norm of $Kf$, hence we need to compute the energy of $w$. Multiplying \eqref{w eq} by $2c^{-2}w_t$ and integrating over $(0,T)\times\Omega$ we obtain
\begin{align*}
E_{\Omega}(w,0) &= E_{\Omega}(w,T) +2\int_{[0,T]\times\Omega}ac^{-2}\bar{u}_tw_tdxdt + 2\int_{[0,T]\times\Omega}ac^{-2}v_tw_tdxdt\\
&\hspace{1em}+ 2\int_{[0,T]\times\Omega}c^{-2}(\Psi*\partial_t\bar{u})w_tdxdt +2\int_{[0,T]\times\Omega}c^{-2}\big(\Psi\tilde{*}\partial_tv\big)w_tdxdt\\
&= E_{\Omega}(w,T) +2\int_{[0,T]\times\Omega}ac^{-2}|\bar{u}_t|^2dxdt - 2\int_{[0,T]\times\Omega}ac^{-2}|v_t|^2dxdt\\
&\hspace{1em}+ 2\int_{[0,T]\times\Omega}c^{-2}(\Psi*\partial_t\bar{u})\partial_t\bar{u}dxdt -  2\int_{[0,T]\times\Omega}c^{-2}(\Psi\tilde{*}\partial_tv)\partial_tvdxdt\\
&\hspace{1em} -2\int_{[0,T]\times\Omega}c^{-2}\left(\Psi*\partial_t\bar{u}\right)\partial_tvdxdt + 2\int_{[0,T]\times\Omega}c^{-2}\big(\Psi\tilde{*}\partial_tv\big)\partial_t\bar{u}dxdt.
\end{align*}
Neglecting the integration in the spatial variable in the last two terms for a moment, we can use the identity \eqref{identity1} which makes them cancel each other out. Furthermore, it follows from the same identity and Condition \eqref{hyp_att} on the kernels (which guarantees positive-definiteness) that
$$\int_{[0,T]\times\Omega}c^{-2}(\Psi\tilde{*}\partial_tv)\partial_tvdxdt = \int_{[0,T]\times\Omega}c^{-2}(\Psi*\partial_tv)\partial_tvdxdt\geq 0.$$
In consequence we get
\begin{equation}\label{eng1}
\begin{aligned}
E_{\Omega}(w,0) \leq E_{\Omega}(w,T) &+2\int_{[0,T]\times\Omega}ac^{-2}|\bar{u}_t|^2dxdt\\
&+ 2\int_{[0,T]\times\Omega}c^{-2}(\Psi*\partial_t\bar{u})\partial_t\bar{u}dxdt.
\end{aligned}
\end{equation}
The choice of the time reversal system \eqref{w eq} helps to minimize the total energy in the dynamic satisfied by the error function $w$ in a similar way as the functions $\phi$ helps to minimize the energy of $w$ at time $T$. Indeed, by integration by parts we have that
$$(\bar{u}^T-\phi,\phi)_{H_D(\Omega)} = -\int_{\Omega}(\bar{u}^T - \phi)\Delta \phi dx + \int_{\partial\Omega}(\bar{u}^T-\phi)\partial_\nu\phi dS = 0,$$
therefore
\begin{equation}\label{eng2}
E_\Omega(w(T)) = \|\bar{u}^T-\phi\|^2_{H_D(\Omega)}+\|\bar{u}^T_t\|^2_{L^2(\Omega)} = E_{\Omega}(\bar{u}(T)) - \|\phi\|_{H_D(\Omega)}^2.
\end{equation}
From the above relations \eqref{eng1} and \eqref{eng2}, we deduce
\begin{equation}\label{eng3}
\|Kf\|^2_{L^2(\Omega;c^{-2}dx)}\leq E_{\Omega}(w,0)\leq \mathcal{E}_{\Omega}(\bar{u},T),
\end{equation}
where recall the term in the right hand side is the extended energy functional associated to \eqref{TAT2} and defined in \eqref{ext_eng}.
By conservation of the extended energy in $\RR^n$, 
\begin{equation}\label{eng4}
\mathcal{E}_{\Omega}(\bar{u},T) = \mathcal{E}_{\RR^n}(\bar{u},T) - E_{\Omega^c}(\bar{u},T) = \|f\|^2_{L^2(c^{-2}dx)} - E_{\Omega^c}(\bar{u},T).
\end{equation}
The conclusion of the theorem follows from the next proposition which is known to hold when there is no integral term. 
\begin{proposition}\label{prop} There is $C>0$ so that for all $f\in L^2(\Omega;c^{-2}dx)$ and $\bar{u}$ solutions of \eqref{eq_bar_u},
$$\|f\|^2_{L^2(\Omega;c^{-2}dx)} \leq CE_{\Omega^c}(\bar{u},T).$$
\end{proposition}
An inequality of this form was first proved in \cite[Proposition 5.1]{TATbrain} (see (5.15) in the same article) for the case of the unattenuated wave equation, and later extended to the damped case in \cite[Proposition 2]{2016Palacios}, requiring a larger lower bound for the measurement time though. Such estimate is obtained by microlocalizing near the singularities and studying how their energy is transmitted across the boundary provided they hit the boundary in a transversal way. By considering strictly convex domains we can be sure that all singularities meet that requirement. When there is no damping coefficient the analysis of the singularities can be decoupled to those following the positive sound speed and negative sound speed. The time needed then for the estimate to hold equals the time needed to get at least one signal from each singularity of the initial condition, this is $T>\frac{1}{2}T_1(\Omega)$. In contrast, the appearance of a damping term makes no longer possible such microlocal decoupling, and therefore it makes necessary to wait until both signals, issued from every singularity of the initial condition, reach the boundary, or in other words $T>T_1(\Omega)$.

Let's prove the above proposition. Denote by $\bar{U}(x,t)$ the solution of the damped wave equation 
\begin{equation}\label{TAT_unatt}
\left\{\begin{array}{ll} (\partial^2_t+ a\partial_t- c^{2}\Delta + b)\bar{U}(t,x) = 0,&(t,x)\in(0,T)\times\RR^{n}\\ \bar{U}|_{t=0}=0,\\ \bar{U}_t|_{t=0}=f.\end{array}\right.
\end{equation}
Denoting $\ff = [0,f]\in \mathcal{H}(\Omega)$, from the paragraph above follows there is $C>0$ so that
$$\|f\|^2_{L^2(\Omega;c^{-2}dx)} = \|\ff\|^2_{\mathcal{H}(\Omega)}\leq CE_{\Omega^c}(U,T).$$
Furthermore, defining $\bar{W} = \bar{U}-\bar{u}$ we obtain
\begin{align*}
\|f\|^2_{L^2(\Omega;c^{-2}dx)}&\leq C\left(E_{\Omega^c}(\bar{u},T)+E_{\Omega^c}(\bar{W},T)\right)
\end{align*}
and letting $\bar{\uu}(t) = [\bar{u}(t),\bar{u}_t(t)]$, $\bar{\text{\bf W}}(t) = [\bar{W}(t),\bar{W}_t(t)]$, the previous inequality implies
\begin{align*}
\|f\|_{L^2(\Omega;c^{-2}dx)}\leq C\|\bar{\uu}(T)\|_{H^1(\Omega^c)\otimes L^2(\Omega^c)} + C\|\bar{\text{\bf W}}(T)\|_{H^1(\Omega^c)\otimes L^2(\Omega^c)},
\end{align*}
where the error function $\bar{W}$ satisfies the IVP
\begin{equation}\label{TAT_error}
\left\{\begin{array}{ll} (\partial^2_t+a\partial_t-c^{2}\Delta+b)\bar{W} = \Psi*\partial_t\bar{u},&(t,x)\in(0,T)\times\RR^{n}\\ \bar{W}|_{t=0}=0,\\ \bar{W}_t|_{t=0}=0.\end{array}\right.
\end{equation}
We claim the bounded map $L^2(\Omega;c^{-2}dx)\ni f\mapsto \bar{\uu}(T)\in H^1(\Omega^c)\otimes L^2(\Omega^c)$ is injective. In fact, it can be decomposed as the composition of two injective bounded maps, the first one being the observation operator $\bar{\Lambda}_\Psi$, which is injective since \eqref{eq_bar_u} is equivalent (following the computation in \eqref{Phi_Psi}) to a system of the form \eqref{eq_uniq} where the method used to prove Theorem \ref{uniqueness} (resp. Theorem \ref{uniqueness2}) can be applied, and our choice of $T>\frac{1}{2}T_1\geq T_0$ (resp. $T>\alpha^{-1}c^{-1}_0R_\Omega\geq \frac{1}{2}T_1$). The second map is the exterior IBVP map that takes Dirichlet boundary data $\bar{h}\in H^1_{(0)}([0,T]\times\partial\Omega)$ to $\bar{\vv}(T)\in H^1(\Omega^c)\otimes L^2(\Omega^c)$, where $\bar{v}$ solves:
\begin{equation}\label{exterior_IBVP}
\left\{\begin{array}{ll} (\partial^2_t - c^{2}\Delta) \bar{v}(t,x) = 0,&(t,x)\in(0,T)\times\RR^{n}\setminus \overline{\Omega}\\ \bar{v}|_{t=0}=0,\\ \partial_t\bar{v}|_{t=0}=0,\\
\bar{v}|_{[0,T]\times\partial\Omega} = \bar{h}.\end{array}\right.
\end{equation}
To see the injectivity of the latter map, consider $\bar{h}\in H^1_{(0)}([0,T]\times\partial\Omega)$ such that $\bar{v}(T) = \bar{v}_t(T) = 0$, with $\bar{v}$ solution of \eqref{exterior_IBVP}. By domain of dependence and reversibility in time of the exterior problem, we have that $\bar{v}$ vanishes in $\{(t,x)\in (0,\infty)\times\RR^n\backslash\bar{\Omega}: \text{dist}_e(x,\partial\Omega)>t\}$ and also in $\{(t,x)\in (0,\infty)\times\RR^n\backslash\bar{\Omega}: \text{dist}_e(x,\partial\Omega)>|T-t|\}$. Therefore 
$$\bar{v}=0\quad\text{in}\quad \{(t,x)\in (0,3T/2)\times\RR^n\backslash\bar{\Omega}: \text{dist}_e(x,\partial\Omega)>T/2\}.$$
\begin{figure}
\centering
\includegraphics[width=0.4\textwidth]{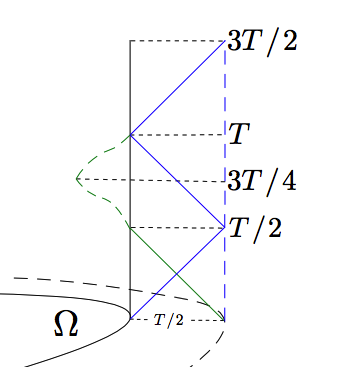}
\vspace{-.7em}
\caption{Unique continuation from points in $\{x\in\RR^n\backslash\bar{\Omega}:\text{dist}_e(x,\partial\Omega)>T/2\}$ implies null Cauchy data on $(T/2,T)\times\partial\Omega$.}
\label{fig:1}
\end{figure}
Applying Tataru's unique continuation theorem on any $p\in \{\text{dist}_e(x,\partial\Omega)>T/2\}$, we deduce that 
$$\bar{v} = 0 \quad\text{in}\quad (\RR^n\backslash\bar{\Omega})\cap\{(t,x)\in (0,\infty)\times\RR^n : |x-p| + |t-3T/4|<3T/4\},$$
which implies that $\bar{h}$ vanishes for $t\in(T/2,T)$ (see Figure \ref{fig:1}). We can now apply the same argument replacing $T$ by $T/2$ and get that $\bar{h}$ is null in the interval $(T/4,T/2)$. Iterating this process we finally conclude that $\bar{h} = 0$ for all $t\in(0,T)$.

Our second claim is that the map 
$$L^2(\Omega;c^{-2}dx)\ni f\mapsto \bar{\WW}(T)\in H^1(\Omega^c)\otimes L^2(\Omega^c)$$ 
is compact. It is in fact a composition of the bounded maps 
$$L^2(\Omega;c^{-2}dx)\ni f\mapsto \bar{u}_t\in L^2((0,T);L^2(\RR^n)), \quad \bar{u}_t\mapsto \bar{\WW}(T)\in H^2(\Omega^c)\otimes H^1(\Omega^c)$$ 
and the compact embedding 
$$H^2(\Omega^c)\otimes H^1(\Omega^c)\hookrightarrow H^1(\Omega^c)\otimes L^2(\Omega^c).$$ 
The continuity of the second map for those Sobolev spaces is due to \cite[\S 7.2.3 Theorem 5]{evans1998partial} since denoting $F := \Psi*\bar{u}_t$, then $F,F_t\in L^2((0,T);L^2(\Omega^c))$. 
It follows from \cite[Proposition V.3.1]{Taylor} that for a different constant
$$\|f\|_{L^2(\Omega;c^{-2}dx)}\leq C\|\bar{\uu}(T)\|_{H^1(\Omega^c)\otimes L^2(\Omega^c)}.$$
The proposition is then proved by recalling the finite speed of propagation and applying Poincare's inequality on a large ball minus $\Omega$. \hfill $\blacksquare$\\

We conclude the proof of Theorem \ref{teo:Nseries} by joining \eqref{eng3}, \eqref{eng4} and  Proposition \ref{prop}, hence for some $C>1$,
\begin{align*}
\|Kf\|^2_{L^2(\Omega;c^{-2}dx)}&\leq \|f\|^2_{L^2(\Omega;c^{-2}dx)} - E_{\Omega^c}(u,T)\\
&\leq \|f\|^2_{L^2(\Omega;c^{-2}dx)} - C^{-2}\|f\|^2_{L^2(\Omega;c^{-2}dx)} \\
&\leq (1-C^{-2})\|f\|^2_{L^2(\Omega;c^{-2}dx)}.
\end{align*}
\end{proof}

\section*{Appendix}

\subsection*{Well-posedness of the direct problem}
For the existence of solutions we follows the proof of \cite[Theorem 2.1]{1970Dafermos}.
Let's assume without lost of generality that $u_0 = 0$. For a fixed $t_0\in (0,T]$ let
$$\mathcal{E}_{t_0} = \{v(t)|v(t)\in C^\infty([0,t_0];H^1_0(U)), v(0) = 0\},$$
with two inner product given by
$$(v,w)_{1}:= \int^{t_0}_0\big\{(v_t(t),w_t(t)) \rangle + ( \nabla v(t),\nabla w(t))\big\}dt$$
and 
$$(v,w)_{2}:= (v,w)_1 + t_0 (v_t(0), w_t(0)),$$
and respective norms $\|\cdot\|_1$ and $\|\cdot\|_{2}$. Let $F_{t_0}$ be the completion of $\mathcal{E}_{t_0}$ under the norm $\|\cdot\|_1$. It can be proved, for instance by Stone-Weierstrass, that $u\in F_{t_0}$ is a generalized solution in the interval $[0,t_0]$ if and only if
\begin{equation}\label{weak_sol}
\mathcal{B}(u,v) = \mathcal{D}(f,v) + t_0(c^{-2}u_1,v_t(0))_{L^2(U)},\quad\forall v\in \mathcal{E}_{t_0},
\end{equation}
where
\begin{eqnarray*}
\mathcal{B}(u,v) &=&\int^{t_0}_0(t-t_0)\Big[ (c^{-2}u_t(t),v_{tt}(t)) - (\nabla u(t),\nabla v(t))- (c^{-2}au_t(t),v_t(t))\\
&& - (c^{-2}bu(t),v_t(t)) - \int^t_0(c^{-2}\Phi(t-\tau)u_\tau(\tau),u_t(t))d\tau\Big]dt\\
&& + \int^{t_0}_0(c^{-2}u_t(t),v_t(t))dt,\nonumber\\
\mathcal{D}(f,v) &=&-\int^{t_0}_0(t-t_0)(c^{-2}f(t),v_t(t))dt.
\end{eqnarray*}
where \eqref{weak_sol} is obtained by using the test function $(t-t_0)v_t(t)$ with $v\in \mathcal{E}_{t_0}$, in \eqref{def:weak_sol}. Notice that applying integration by parts we get that the bilinear form $\mathcal{B}$ satisfies that for all $v\in \mathcal{E}_{t_0}$ (recall $v(0) = 0$),
$$
\begin{aligned}
\mathcal{B}(v,v) =&\frac{1}{2}\int^{t_0}_0\left[(c^{-2}v_t(t),v_t(t)) + (\nabla v(t),\nabla v(t)) + (c^{-2}bv(t),v(t))\right]dt\\
& - \int^{t_0}_0(t-t_0)\Big[ (c^{-2}av_t(t),v_t(t)) - (c^{-2}\Phi(0)v(t),v(t))\\
&\hspace{5em}-\int^t_0(c^{-2}\Phi(t-s)v(s),v(t))ds \Big]dt\\
&+ \frac{t_0}{2}(c^{-2}v_t(0),v_t(0)).
\end{aligned}
$$
Therefore, recalling that $0<c_{0}\leq c\leq c^{-1}_0$, we bound from bellow and choosing $t_0>0$ small enough and using Poincare's inequality we get
$$
\begin{aligned}
\mathcal{B}(v,v) \geq& \frac{1}{2}\min\{1,c^2_0\}\|v\|^2_{2} - Ct_0\big([\|a\|_\infty + \|\Phi(0)\|_\infty + t_0\|\Phi\|_\infty\big)\|v\|^2_{1} \geq \delta\|v\|^2_{2}.
\end{aligned}
$$
for some $\delta>0$. 

On the other hand, 
$$|\mathcal{D}(f,v)|\leq t_0c^{-2}_0\|f\|_{L^2}\|v\|_2$$
$$|t_0(c^{-2}u_1,v_t(0))| \leq t_0^{1/2}c^{-2}_0\|u_1\|_{L^2}\|v\|_2$$
Then, similarly as in \cite[Chap. III, Theorem 1.1]{1961Lions}, we get the existence of weak solutions on the interval $[0,t_0]$. Iterating this argument for the intervals $[t_0,2t_0]$, $[2t_0,3t_0]$ etc, we conclude the existence on $[0,T]$. The uniqueness follows the same ideas as in \cite[Theorem 2.2]{1970Dafermos}.

\renewcommand{\refname}{\centerline{\large \sc References}}
\bibliographystyle{plain}
{\footnotesize
\bibliography{biblio}}

\begin{thebibliography}{10}

\bibitem{2015AcostaMontalto}
S.~Acosta and C.~Montalto.
\newblock Multiwave imaging in an enclosure with variable wave speed.
\newblock {\em Inverse Problems}, 31(6):065009, 12, 2015.

\bibitem{2016AcostaMontalto}
S.~Acosta and C.~Montalto.
\newblock Photoacoustic imaging taking into account thermodynamic attenuation.
\newblock {\em Inverse Problems}, 32(11):115001, 2016.

\bibitem{Agranovsy-2007}
M.~Agranovsky, P.~Kuchment, and L.~Kunyansky.
\newblock {\em On reconstruction formulas and algorithms for the thermoacoustic
  and photoacoustic tomography}.
\newblock Photoacoustic imaging and spectroscopy. CRC Press, 2009.

\bibitem{Anastasio-2007}
M.~A. Anastasio, J.~Zhang, D.~Modgil, and P.~J.~La Rivi\`{e}re.
\newblock Application of inverse source concepts to photoacoustic tomography.
\newblock {\em Inverse Problems}, 23(6):S21, 2007.

\bibitem{Bal-2011}
G.~Bal.
\newblock Hybrid inverse problems and internal functionals.
\newblock {\em Inverse Problems and Applications: Inside Out II}, 60:325--368,
  2011.

\bibitem{2007Bukhgeim}
A.~L. Bukhgeim, G.~V. Dyatlov, and G.~Uhlmann.
\newblock Unique continuation for hyperbolic equations with memory.
\newblock {\em J. Inverse Ill-Posed Probl.}, 15(6):587--598, 2007.

\bibitem{2015Caputo}
M.~Caputo and M.~Fabrizio.
\newblock A new definition of fractional derivative without singular kernel.
\newblock {\em Progr. Fract. Differ. Appl}, 1(2):1--13, 2015.

\bibitem{chen2003modified}
W.~Chen and S.~Holm.
\newblock {Modified Szabo's wave equation models for lossy media obeying
  frequency power law}.
\newblock {\em The Journal of the Acoustical Society of America},
  114(5):2570--2574, 2003.

\bibitem{1970Dafermos}
C.~Dafermos.
\newblock An abstract {V}olterra equation with applications to linear
  viscoelasticity.
\newblock {\em J. Differential Equations}, 7:554--569, 1970.

\bibitem{evans1998partial}
Lawrence~C Evans.
\newblock Partial differential equations.
\newblock {\em Graduate Studies in Mathematics}, 19, 1998.

\bibitem{2007Finch}
D.~Finch, M.~Haltmeier, and Rakesh.
\newblock Inversion of spherical means and the wave equation in even
  dimensions.
\newblock {\em SIAM J. Appl. Math.}, 68(2):392--412, 2007.

\bibitem{Finch_Rakesh_2004}
D.~Finch, S.~Patch, and Rakesh.
\newblock Determining a function from its mean values over a family of spheres.
\newblock {\em SIAM J. Math. Anal.}, 35(5):1213--1240, 2004.

\bibitem{Haltmeier2014}
M.~Haltmeier.
\newblock {Universal inversion formulas for recovering a function from
  spherical means}.
\newblock {\em SIAM J. Math. Anal.}, 46(1):214--232, 2014.

\bibitem{Homan}
A.~Homan.
\newblock Multi-wave imaging in attenuating media.
\newblock {\em Inverse Problems and Imaging}, 7(4):1235--1250, 2013.

\bibitem{2009Hristova}
Y.~Hristova.
\newblock Time reversal in thermoacoustic tomography---an error estimate.
\newblock {\em Inverse Problems}, 25(5):055008, 14, 2009.

\bibitem{hristova2008reconstruction}
Y.~Hristova, P.~Kuchment, and L.~Nguyen.
\newblock Reconstruction and time reversal in thermoacoustic tomography in
  acoustically homogeneous and inhomogeneous media.
\newblock {\em Inverse Problems}, 24(5):055006, 2008.

\bibitem{2012PATatt}
C.~Huang, L.~Nie, R.~Schoonover, L.~Wang, and M.~Anastasio.
\newblock Photoacoustic computed tomography correcting for heterogeneity and
  attenuation.
\newblock {\em Journal of Biomedical Optics}, 17(6):061211--1--061211--5, 2012.

\bibitem{HuangWaNie2013}
C.~Huang, K.~Wang, L.~Nie, L.V. Wang, and M.~Anastasio.
\newblock Full-wave iterative image reconstruction in photoacoustic tomography
  with acoustically inhomogeneous media.
\newblock {\em Medical Imaging, IEEE Transactions on}, 32(6):1097--1110, 2013.

\bibitem{2011Ammari}
R.~Kowar and O.~Scherzer.
\newblock {\em Attenuation Models in Photoacoustics}.
\newblock ed. H. Ammari, Springer Berlin Heidelberg, Berlin, Heidelberg, 2012.

\bibitem{Kuchment2011}
P.~Kuchment and L.~Kunyansky.
\newblock Mathematics of photoacoustic and thermoacoustic tomography.
\newblock In Otmar Scherzer, editor, {\em Handbook of Mathematical Methods in
  Imaging}, pages 817--865. Springer New York, 2011.

\bibitem{Kunyansky2011}
L.~Kunyansky.
\newblock {Reconstruction of a function from its spherical (circular) means
  with the centers lying on the surface of certain polygons and polyhedra}.
\newblock {\em Inverse Probl.}, 27(2):025012, 2011.

\bibitem{2015Nguyen}
L.~Kunyansky L.~V.~Nguyen.
\newblock A dissipative time reversal technique for photo-acoustic tomography
  in a cavity.
\newblock {\em SIAM Journal on Imaging Sciences}, 9(2):748--769, 2016.

\bibitem{1961Lions}
J.-L. Lions.
\newblock {\em \'{E}quations diff\'erentielles op\'erationnelles et probl\`emes
  aux limites}.
\newblock Die Grundlehren der mathematischen Wissenschaften, Bd. 111.
  Springer-Verlag, Berlin-G\"ottingen-Heidelberg, 1961.

\bibitem{2009Kunyansky}
P.~Kuchment M.~Agranovsky and L.~Kunyansky.
\newblock On reconstruction formulas and algorithms for the thermoacoustic and
  photoacoustic tomography.
\newblock {\em Photoacoustic imaging and spectroscopy. CRC Press}, chapter
  8:89--101, 2009.

\bibitem{1972MacCamy}
R.~C. MacCamy and J.~S. Wong.
\newblock Stability theorems for some functional equations.
\newblock {\em Trans. Amer. Math. Soc.}, 164:1--37, 1972.

\bibitem{modgil2012}
D.~Modgil, B.~Treeby, and P.~La~Rivi{\`e}re.
\newblock Photoacoustic image reconstruction in an attenuating medium using
  singular-value decomposition.
\newblock {\em Journal of biomedical optics}, 17(6):0612041--0612048, 2012.

\bibitem{Natterer2012}
F.~Natterer.
\newblock {Photo-acoustic inversion in convex domains}.
\newblock {\em Inverse Probl. Imaging}, 6(2):315--320, 2012.

\bibitem{2016Oksanen}
L.~Oksanen O.~Chervova.
\newblock Time reversal method with stabilizing boundary conditions for
  photoacoustic tomography.
\newblock {\em arXiv:1605.07817}, 2016.

\bibitem{2016Palacios}
B.~Palacios.
\newblock Reconstruction for multi-wave imaging in attenuating media with large
  damping coefficient.
\newblock {\em Inverse Problems}, 32(12):125008, 2016.

\bibitem{Palamodov2012}
V.~Palamodov.
\newblock A uniform reconstruction formula in integral geometry.
\newblock {\em Inverse Problems}, 28(6):065014, 2012.

\bibitem{2011Roitner}
H.~Roitner and P.~Burgholzer.
\newblock Efficient modeling and compensation of ultrasound attenuation losses
  in photoacoustic imaging.
\newblock {\em Inverse Problems}, 27(1):015003, 15, 2011.

\bibitem{2006Patch}
M.~Haltmeier S.~K.~Patch.
\newblock Thermoacoustic tomography - ultrasound attenuation effects.
\newblock {\em IEEE Nuclear Science Symposium Conference Record}, 4:2604--2606,
  2006.

\bibitem{2016IPwave_memory}
L.~{\v{S}}eliga and M.~Slodi{\v{c}}ka.
\newblock An inverse source problem for a damped wave equation with memory.
\newblock {\em J. Inverse Ill-Posed Probl.}, 24(2):111--122, 2016.

\bibitem{TAT}
P.~Stefanov and G.~Uhlmann.
\newblock Thermoacoustic tomography with variable sound speed.
\newblock {\em Inverse Problems}, 25(7):075011, 16, 2009.

\bibitem{TATbrain}
P.~Stefanov and G.~Uhlmann.
\newblock Thermoacoustic tomography arising in brain imaging.
\newblock {\em Inverse Problems}, 27(4):045004, 26, 2011.

\bibitem{UhlmannStefanov2012InsideOut}
P.~Stefanov and G.~Uhlmann.
\newblock Multi-wave methods via ultrasound.
\newblock {\em Inverse Problems and Applications, Inside Out II, MSRI
  Publications}, 60:271--323, 2012.

\bibitem{2013StefanovUhlmann}
P.~Stefanov and G.~Uhlmann.
\newblock Recovery of a source term or a speed with one measurement and
  applications.
\newblock {\em Trans. Amer. Math. Soc.}, 365(11):5737--5758, 2013.

\bibitem{Uhlmann2016bdry_rig}
P.~Stefanov, G.~Uhlmann, and A.~Vasy.
\newblock Boundary rigidity with partial data.
\newblock {\em J. Amer. Math. Soc.}, 29(2):299--332, 2016.

\bibitem{2015StefanovYang}
P.~Stefanov and Y.~Yang.
\newblock Multiwave tomography in a closed domain: averaged sharp time
  reversal.
\newblock {\em Inverse Problems}, 31(6):065007, 23, 2015.

\bibitem{1994Szabo}
T.~Szabo.
\newblock Time domain wave equations for lossy media obeying a frequency power
  law.
\newblock {\em The Journal of the Acoustical Society of America},
  96(1):491--500, 1994.

\bibitem{Tataru1}
D.~Tataru.
\newblock Unique continuation for solutions to {PDE}'s; between {H}\"ormander's
  theorem and {H}olmgren's theorem.
\newblock {\em Comm. Partial Differential Equations}, 20(5-6):855--884, 1995.

\bibitem{Tataru2}
D.~Tataru.
\newblock Unique continuation for operators with partially analytic
  coefficients.
\newblock {\em J. Math. Pures Appl. (9)}, 78(5):505--521, 1999.

\bibitem{Taylor}
M.~Taylor.
\newblock Pseudodifferential operators.
\newblock {\em (Princeton mathematical series ; 34). Princeton, N.J.: Princeton
  University Press.}, 1981.

\bibitem{Tittelfitz-2012}
J.~Tittelfitz.
\newblock Thermoacoustic tomography in elastic media.
\newblock {\em Inverse Problems}, 28(5):055004, 2012.

\bibitem{2010TreebyCox}
B.~Treeby and B.~T. Cox.
\newblock {Modeling power law absorption and dispersion for acoustic
  propagation using the fractional Laplacian}.
\newblock {\em The Journal of the Acoustical Society of America},
  127(5):2741--2748, 2010.

\bibitem{treeby2010}
B.~Treeby, E.~Zhang, and B.~Cox.
\newblock Photoacoustic tomography in absorbing acoustic media using time
  reversal.
\newblock {\em Inverse Problems}, 26(11):115003, 2010.

\bibitem{WangAnastasio2012}
K.~Wang and M.~Anastasio.
\newblock {A simple Fourier transform-based reconstruction formula for
  photoacoustic computed tomography with a circular or spherical measurement
  geometry}.
\newblock {\em Physics in Medicine and Biology}, 57(23):N493, 2012.

\end{thebibliography}
\end{document}